\documentclass[11pt,bezier]{article}

\usepackage{amsmath,amssymb,amsfonts,amsthm,mathtools,mathabx,cite}

\newtheorem{thm}{Theorem}
\newtheorem{cor}{Corollary}
\newtheorem{lem}{Lemma}
\newtheorem{conj}{Conjecture}

\date{}
\title{\bf Chromatic Number and Dichromatic Polynomial of Digraphs}
\author{
\bf S. Akbari$^a$,
A.H. Ghodrati$^{a,c}$,
A. Jabalameli$^b$,
M. Saghafian$^a$
\footnote{E-mail addresses: s$\_$akbari@sharif.edu (S. Akbari),
ghodrati$\_$ah@mehr.sharif.ir (A.H. Ghodrati),
jabal.farsh@gmail.com (A. Jabalameli)
saghafian$\_$morteza@mehr.sharif.ir (M. Saghafian)}\\
$^a$\small Department of Mathematical Sciences, 
\small Sharif University of Technology,\\
$^b$\small Department of Computer Engineering,
\small Sharif University of Technology,\\
$^c$\small School of Mathematics, Institute for Research in
Fundamental Sciences(IPM),
\small 9395-5746,\\
\small Tehran, Iran.}

\begin{document}
\maketitle

\begin{abstract}
Let $G$ be a graph of order $n$. It is well-known that 
$\alpha(G)\geq \sum_{i=1}^n \frac{1}{1+d_i}$, where $\alpha(G)$ is the 
independence number of $G$ and $d_1,\ldots,d_n$ is the degree sequence of $G$.
We extend this result to digraphs by showing that if $D$ is a digraph with $n$
vertices, then 
\\$ \alpha(D)\geq \sum_{i=1}^n \left( \frac{1}{1+d_i^+} + \frac{1}{1+d_i^-}
  - \frac{1}{1+d_i}\right)$, where $\alpha(D)$ is the maximum size of an 
  acyclic vertex set of $D$.
Golowich proved that for any digraph $D$, 
$\chi(D)\leq \lceil \frac{5k}{2} \rceil+2$, 
where $k=max(\Delta^+(D),\Delta^-(D))$. We give a short and simple proof for this result.
Next, we investigate the chromatic number of tournaments
and determine the unique tournament such that for every integer $k>1$,
the number of proper $k$-colorings of that tournament is maximum among
all strongly connected tournaments with the same number of vertices. Also,
we find the chromatic polynomial of the strongly connected tournament with 
the minimum number of cycles.

\vskip 3mm

\noindent{\bf Keywords:} Digraph, Chromatic Number, Dichromatic Polynomial,
 Independence Number.

\vskip 3mm

\noindent{\bf 2010 Mathematics Subject Classification:} 05C15, 05C20.

\end{abstract}
\section{Introduction}
Throughout this paper all graphs and digraphs are loopless.
The vertex set and the edge set of a digraph $D$ are denoted by $V(D)$ and $E(D)$
respectively. This digraph is called \emph{strict} if for any two 
distinct vertices $u$ and $v$ of $D$, there is at most one edge from $u$ to $v$.
All digraphs in this paper are assumed to be strict. A digraph is called
\emph{digon-free} if it has no directed cycle of length $2$. 
An {\it acyclic} vertex set of a digraph $D$ is a 
subset $S$ of $V(D)$ such that the induced subdigraph of $D$ on $S$ has no 
directed cycle. The maximum size of an acyclic vertex set of $D$ is called
the {\it independence number} of $D$ and is denoted by $\alpha(D)$.
A {\it proper $k$-coloring} of $D$ is a partition of the vertex set of $D$
into at most $k$ acyclic subsets. The minimum integer $k$ for which $D$ has a 
proper $k$-coloring is called the {\it chromatic number} of $D$ and is denoted by
$\chi(D)$. This definition for the coloring of digraphs was first appeared 
in \cite{NEU}. The independence and chromatic number of the undirected graph
$G$ is also denoted by $\alpha(G)$ and $\chi(G)$, respectively.
It turns out that the number of proper $k$-colorings of $D$ is a
polynomial in $k$ \cite{HARUT} which is called the {\it chromatic polynomial} of 
$D$ and is denoted by $P(D;k)$.
In what follows, $N^+(v)$ and $N^-(v)$ denote the set of out-neighbors and the 
set of in-neighbors
of $v$ and their cardinalities are called the {\it in-degree} 
and the {\it out-degree} of 
$v$, respectively.
The {\it girth} of a digraph is defined as the length of its shortest
directed cycle.

Section $2$ is about the independence number of digraphs and some lower bounds
on that number. In particular, we find a lower bound for the independence 
number of digraphs which is similar to a well-known
bound for graphs proved in \cite{CARO1} and \cite{WEI}. Moreover, we find
another lower bound for the independence number of a digraph, 
in terms of its girth and number of induced cycles.

In Section $3$ we investigate various upper bounds on the chromatic number
of digraphs. Note that a trivial upper bound for the chromatic number
of a digraph with underlying graph $G$ is $\chi(G)$. We show that
the equality holds for some digraphs.
In \cite{GOLO} Golowich proved that for any digraph $D$, 
$\chi(D)\leq \lceil \frac{5k}{2} \rceil+2$, 
where $k=max(\Delta^+(D),\Delta^-(D))$.
Also in \cite{CHEN}, an upper bound for the chromatic number of digraphs 
with some forbidden cycle lengths is given. Here we give short and simple
proofs for these two results.
Moreover, we give a new upper bound for the chromatic number in terms of 
the girth.

The dichromatic polynomial of digraphs is the subject of Section $4$.
We give an interpretation for some coefficients of this polynomial, and
determine the tournament which has the maximum number of proper $k$-colorings
for every $k$ between all strongly connected tournaments with $n$ vertices.
\section{Lower Bounds on the Independence Number}

Caro \cite{CARO1} and Wei \cite{WEI} found the following lower bound
for the independence number of a graph.

\begin{thm}{\label{INDMAIN}}
 Let $G$ be a graph and $d_1,\ldots,d_n$ be the degree sequence
 of $G$. Then the following inequality holds:
  \begin{displaymath}
  \alpha(G)\geq \sum_{i=1}^n \frac{1}{1+d_i}
 \end{displaymath}
\end{thm}

A probabilistic proof of this result can be found in \cite{ALON}.
In \cite{CARO2} and \cite{THIELE}, this result
has been generalized to hypergraphs.

Here we prove an analogous inequality for the independence number of  a digraph.

\begin{thm}\label{indep1}
 Let $D$ be a digraph with the vertex set $\{v_1,\ldots,v_n\}$, 
 and let $d_i,d_i^+,d_i^-$, $i=1,\ldots,n$, be the degree,
 the out-degree and the in-degree of $v_i$, respectively. 
 Then the following inequality holds:
 \begin{displaymath}
  \alpha(D)\geq \sum_{i=1}^n \left( \frac{1}{1+d_i^+} + \frac{1}{1+d_i^-}
  - \frac{1}{1+d_i}\right)
 \end{displaymath}
\end{thm}
\begin{proof}
 Choose $\sigma \in S_n$ uniformly at random. For $i=1,\ldots,n$, let $A_i$
 be the event that for $j=1,\ldots,n$, if $v_{\sigma(j)}\in N^+(v_i)$, then
 $\sigma(j)>i$. In other words, if we arrange $v_{\sigma(1)},\ldots,
 v_{\sigma(n)}$ on a line from left to the right, then $A_i$ is the event 
 that all of the out-neighbors of $v_i$ lie on the right of $v_i$. 
 Similarly, let $B_i$ be the event
 that for $j=1,\ldots,n$, if $v_{\sigma(j)}\in N^-(v_i)$, then
 $\sigma(j)>i$. It is not hard to see that:
 $$\Pr(A_i)=\frac{1}{1+d_i^+},\,\,\,\Pr(B_i)=\frac{1}{1+d_i^-},$$
 $$\Pr(A_i\cap B_i)=\frac{1}{1+d_i},$$
 and thus, by the principle of inclusion and exclusion:
 $$\Pr(A_i\cup B_i)=\frac{1}{1+d_i^+} + \frac{1}{1+d_i^-}
  - \frac{1}{1+d_i}.$$
  Now, let $X_i$ be the indicator random variable of $A_i\cup B_i$ (that is,
  $X_i=1$ if $A_i\cup B_i$ occurs and $X_i=0$ otherwise). If $S_\sigma$ is 
  the set of
  all $1\leq i\leq n$ such that $A_i\cup B_i$ occurs, then 
  $|S_\sigma|=\sum_{i=1}^nX_i$. So we have the following,
  
  $$\mathbb{E}(|S_\sigma|)=\sum_{i=1}^n\mathbb{E}(X_i)
  =\sum_{i=1}^n\Pr(A_i\cup B_i),$$
  
  which implies that,
  
  $$\mathbb{E}(|S_{\sigma}|)=\sum_{i=1}^n \left( \frac{1}{1+d_i^+} + \frac{1}{1+d_i^-}
  - \frac{1}{1+d_i}\right).$$
 
  Thus, there is a $\sigma \in S_n$ for which
  
  $$|S_\sigma|\geq\sum_{i=1}^n \left( \frac{1}{1+d_i^+} + \frac{1}{1+d_i^-}
  - \frac{1}{1+d_i}\right).$$
  
  Now, to complete the proof, it suffices to show that for 
  each $\sigma$, $\{v_i: i\in S_\sigma\}$ is an acyclic subset of 
  the vertices of $D$. 
  Assume on the contrary that there is a cycle $C$ between
  the vertices in $\{v_i: i\in S_\sigma\}$. Let $i$ be the largest 
  integer that $v_i$
  is a vertex of $C$ (that is, $v_i$ is the right-most vertex
  of $C$ in the random order of the vertices of $D$). Since $C$ is a cycle,
  there are two vertices $v_{\sigma(j)}$ and $v_{\sigma(k)}$ in $V(C)$
  such that $v_{\sigma(j)}\in N^+(v_i)$ and 
  $v_{\sigma(k)}\in N^-(v_i)$. Since $v_{\sigma(j)}\in N^+(v_i)$ and
  $\sigma(j)<i$, $A_i$ does not occur, and similarly, 
  since $v_{\sigma(k)}\in N^-(v_i)$
  and $\sigma(k)<i$, $B_i$ does not occur. But since $i\in S_\sigma$, at least
  one of the $A_i$ and $B_i$ should occur, which is a contradiction.
\end{proof}

Using this theorem, one can obtain some lower bounds for the independence
number of a digraph which is completely
determined by the underlying graph.

\begin{cor}\label{indep2}
 Let $D$ be a digraph with $n$ vertices and $m$ edges, and let
 $d_1,\ldots,d_n$ be the degrees of vertices of the underlying graph
 of $D$. Then
 \begin{itemize}
  \item[(i)] $\alpha(D)\geq{\displaystyle \sum_{i=1}^n
  \frac{3d_i+2}{(d_i+1)(d_i+2)}}$,
  \item[(ii)] If $D$ has no isolated vertices and $k=\frac{m}{n}$, then 
  $\alpha(D)\geq \frac{n}{\frac{2k}{3}+1}$
 \end{itemize}
\end{cor}

\begin{proof}

(i) If $d_i^+$ and $d_i^-$ denote the out-degree and the in-degree
  of the $i$th vertex, then
    $$\frac{1}{1+d_i^+}+\frac{1}{1+d_i^-}\geq\frac{2}{1+\frac{d_i^++d_i^-}{2}}
    =\frac{4}{2+d_i}$$
    and applying Theorem \ref{INDMAIN} yields the first part.
(ii) For any real number $x\geq 1$ define: 
    $$f(x)=\frac{4}{x+2}-\frac{1}{x+1}=
    \frac{3x+2}{(x+1)(x+2)}.$$
    Since for $x\geq 1$
    $$f''(x)=\frac{8}{(x+2)^3}-\frac{2}{(x+1)^3}\geq0,$$
    $f$ is convex and thus,
    $$\sum_{i=1}^n\frac{3d_i+2}{(d_i+1)(d_i+2)}=\sum_{i=1}^nf(d_i)
    \geq nf\left(\frac{\sum_{i=1}^nd_i}{n}\right)=nf(\frac{2m}{n}).$$
    This inequality together with the previous part, gives the following,
    $$\alpha(D)\geq n\frac{3k+1}{(k+1)(2k+1)} \geq \frac{n}{\frac{2k}{3}+1}$$
    which completes the proof.
\end{proof}

In \cite{HARUT}, it is conjectured that for a digon-free digraph $D$, 
$\chi(D) \leq \lceil \frac{\tilde{\Delta}}{2} \rceil +1$, where 
$\tilde{\Delta}=max\{\sqrt{d^+(v)d^-(v)} | v \in V(D)\}$. If this conjecture is
true, then $\alpha(D)\geq \frac{n}{\lceil \frac{\tilde{\Delta}}{2} \rceil +1}$.
In the next corollary, a weaker result about the independence number is proved.

\begin{cor}\label{indep3}
 Let $D$ be a digon-free digraph. Then
 $\alpha(D) \geq \frac{n}{\frac{2\tilde{\Delta}}{3} +1}$.
\end{cor}

\begin{proof}
 Let $V(D)=\{v_1,\ldots,v_n\}$ and $d^+_i=d^+(v_i)$, $d^-_i=d^-(v_i)$ and 
 $d_i=d^+_i+d^-_i$, $p_i=d^+_id^-_i$. We have the following,
 \begin{eqnarray*}
  \frac{1}{1+d^+_i}+\frac{1}{1+d^-_i}-\frac{1}{1+d_i}&=&
  \frac{2+d_i}{1+d_i+p_i}-\frac{1}{1+d_i}.
 \end{eqnarray*}
 
 It can be seen that the above expression is an increasing function of $d_i$
 (for a fixed $p_i$), thus, since $d_i\geq 2\sqrt{p_i}$,
 \begin{eqnarray*}
  \frac{1}{1+d^+_i}+\frac{1}{1+d^-_i}-\frac{1}{1+d_i}&\geq&
  \frac{2+2\sqrt{p_i}}{1+2\sqrt{p_i}+p_i}-\frac{1}{1+2\sqrt{p_i}}\\&=&
  \frac{1+3\sqrt{p_i}}{(1+\sqrt{p_i})(1+2\sqrt{p_i})}.
 \end{eqnarray*}
 
Now, the above expression is a decreasing function of $p_i$, and since
$\tilde{\Delta}=max\{\sqrt{p_i} | i=1,\ldots,n\}$, we find the following,
 \begin{eqnarray*}
  \frac{1}{1+d^+_i}+\frac{1}{1+d^-_i}-\frac{1}{1+d_i}&\geq&
  \frac{1+3\tilde{\Delta}}{(1+\tilde{\Delta})(1+2\tilde{\Delta})}\\&\geq&
  \frac{1}{ \frac{2\tilde{\Delta}}{3} +1}.
 \end{eqnarray*}
 
This inequality and the Corollary \ref{indep2} gives the result.
\end{proof}

In the next theorem we give a lower bound on the independence number of a
digraph in terms of its girth and the number of induced cycles. Note that
a subset of vertices of a digraph is acyclic if and only if it 
contains no induce cycles.
\begin{thm}
 Let $D$ be a digraph with $n$ vertices and girth $g$, and let $t$ be the 
 number of induced directed cycles of $D$. If $tg\geq n$ then
 $$\alpha(D)\geq\frac{g-1}{g}\left(\frac{n^g}{tg}\right)^{\frac{1}{g-1}}$$
\end{thm}

\begin{proof}
 Let $C_1,\ldots,C_t$ be all of the induced cycles in $D$.
 Let $p$ be any number in $[0,1]$. Choose a random subset $S$ of $V(D)$,
 such that for each $v\in V(D)$, $\Pr(v\in S)=p$. 
 Let $Y$ be the number of induced cycles in $D[S]$.
 For each $i=1,\ldots,t$, the probability that $V(C_i)\subseteq S$ is
 $p^{|V(C_i)|}$, so
 $$E(Y)=\sum_{i=1}^tp^{|V(C_i)|}\leq tp^g.$$
 Thus $E(|S|-Y|)\geq np-tp^g$. So there is a subset $S$ of vertices such that
 after removing one vertex of each induced cycle in $D[S]$, at least
 $np-tp^g$ vertices remain. But the remaining vertices constitute an acyclic
 set and so we have $\alpha(D)\geq np-tp^g$, for each $p\in[0,1]$.
 Since $tg\geq n$ we can set $p=\left(\frac{n}{tg}\right)^{\frac{1}{g-1}}$
 which gives the desired result.
\end{proof}

\begin{cor}
 If $T$ is a tournament with $n$ vertices and $t$ directed triangles
 such that $n\geq 3t$, then
 $\alpha(T)\geq \displaystyle \frac{2}{3}n\sqrt{\frac{n}{3t}}$
\end{cor}

\begin{proof}
 In a tournament, induced directed cycles are directed triangle and the result
 follows from the previous theorem.
\end{proof}

\section{Bounds on the Chromatic Number}
There are some known bounds for the chromatic number of digraphs, see 
\cite{HARUT,HARUT2,HARUT3,MOHAR,GOLO,CHEN}. Here we show new bounds for 
the chromatic number and give simpler proofs for some known results. 

It is obvious that if $D$ is any orientation of the graph $G$ with
chromatic number $k$, then $\chi(D)\leq k$. We show that this bound 
is tight for every $k$.
\begin{lem}\label{AWESOME}
 Let $n\geq 10$ and $t\geq 3\log_2n$ be two positive integers. 
 Then there exists an orientation of $K_{n,n}$ such that each 
 of its $K_{t,t}$-subdigraph has a
 directed cycle.
\end{lem}
 
 \begin{proof}
  Let $X,Y$ be two parts of $K_{n,n}$. 
  Let $D$ be a random orientation of $G$. For each $I\subseteq X$ and
  $J\subseteq Y$ with $|I|=|J|=t$, let $A_{I,J}$ be the event that the induced
  subdigraph of $D$ on $I\cup J$ is acyclic. Since an acyclic orientation
  of a digraph gives an ordering of its vertices, we find that
  $$\Pr(A_{I,J})\leq \frac{(2t)!}{2^{t^2}}.$$
  So
  $$\Pr(\bigcup_{I,J}A_{I,J})\leq 
  \binom{n}{t}^2 \frac{(2t)!}{2^{t^2}}<
  \frac{n^{2t}(2t)!}{2^{t^2}(t!)^2}\leq
  \frac{n^{2t}t^t}{2^{t^2}}.$$
  On the other hand, by the assumption
  $$ \frac{n^{2t}t^t}{2^{t^2}}\leq
  \frac{2^{2t^2/3}2^{tlog_2t}}{2^{t^2}}<1.$$
  Thus there is an orientation of $G$ such that no $A_{I,J}$ occurs,
  and the proof is complete.
 \end{proof}
\begin{thm}
 For any positive integer $k$, there exists a graph $G$ with an orientation
 $D$ such that $\chi(D)=\chi(G)=k$.
\end{thm} 

\begin{proof}
Let $n,t$ be two positive integers that satisfy $n>(k-1)t$ and
$t\geq 3\log_2n$.
 Let $G$ be the complete $k$-partite graph with parts
 $X_1,\ldots,X_k$ such that the size of
 each part is $n$. For each $i\neq j$, by
 Lemma \ref{AWESOME}, we can orient the edges between $X_i$
 and $X_j$, in such a way that each $K_{t,t}$-subdigraph of 
 $G[X_i\cup X_j]$ has a directed cycle. Denote the resulting orientation
 of $G$ by $D$. 
 We claim that $\chi(D)=k$. 
 Assume on the contrary
 that $D$ has a proper $(k-1)$-coloring. Since  $n>(k-1)t$, 
 in each part there are at least $t$ vertices that receive the same
 color. For each $i=1\ldots,k$, let $c_i$ be the color that is appeared on 
 maximum number of vertices of $X_i$. Since there are $k$ colors, there
 are $i\neq j$ such that $c_i=c_j$. But there are $A\subseteq X_i$
 and $B\subseteq X_j$ such that $|A|,|B|\geq t$ and all vertices of $A$ 
 have color $c_i$ and all vertices of $B$ have color $c_j$. By our
 choice of orientation, $D[A\cup B]$ has a directed cycle, a contradiction.
\end{proof}

Now, we present another upper bound for the chromatic number of a digraph. 
In \cite{HARUT}, it is proved that for any digraph $D$,
$\chi(D)\leq min\{\Delta^+(D),\Delta^-(D)\} +1$. Also they have
proposed the following conjecture.
\begin{conj}
 For every $k$-regular digraph $D$, 
we have $\chi(D) \leq \lceil \frac{k}{2} \rceil + 1$.
\end{conj}

In \cite{GOLO} they proved that If $D$ is a digraph and 
$k=max\{\Delta^+(D),\Delta^-(D)\}$, then $\chi(D) \leq [\frac{4k}{5}]+2$. 
We provide a shorter and easier proof for this theorem.

\begin{thm}
If $D$ is a digraph and $k=max\{\Delta^+(D),\Delta^-(D)\}$, then 
$\chi(D) \leq [\frac{4k}{5}]+2$.
\end{thm}

\begin{proof}
Let $t=\lceil\frac{2k +1}{5}\rceil$. 
Partition the vertices of $D$ into $t$ parts such that the number of edges
between these parts is maximum possible.
Now, we claim that every vertex at most $4$ neighbors in its part. 
To see this, note that if the number of neighbors of a vertex in 
its part is at least $5$, 
then its degree in one of the other parts should be at most $4$.
So one may change the part of this vertex and make the number of 
edges between parts larger, which is impossible.  
Thus, by Theorem 2.3 of \cite{MOHAR}, we can color the vertices of each part with 
just $2$ colors. 
So we can color all vertices of $D$ with 
$2t=2\lceil\frac{2k +1}{5}\rceil \leq [\frac{4k}{5}]+2$ 
colors such that the subdigraph of each color is acyclic.
\end{proof}

\begin{thm}
Let $D$ be a digraph with girth $g$ and $n$ vertices. Then $\chi(D) \leq [\frac{n-1}{g-1}]+1$.
\end{thm}
\begin{proof}
Let $v_1,v_2,...,v_n$ be the vertices of $D$ and $k=[\frac{n-1}{g-1}]+1$. Color the vertices with $k$ colors from $v_1$ to $v_n$. For each vertex $v_i$, we have at most $[\frac{n-1}{g-1}]$ vertex-disjoint cycles containing $v_i$, so we can use the other color and we are sure that $v_i$ is not in a directed monochromatic cycle.
\end{proof}

In \cite{CHEN} the following result was proved. Here we give a short and simple
proof of this result.
If $T$ is a directed DFS tree of a digraph $D$, then a \emph{back edge}
of $D$ is an edge from a vertex to one of its ancestors.

\begin{thm}
 Let $k \geq 2$ be an integer. If a digraph $D$ contains no directed cycle
 of length $1$ modulo $k$, then $\chi(D) \leq k$.
\end{thm}

\begin{proof}
 Without loss of generality, we can assume that $D$ is strongly connected.
 Let $T$ be a directed  DFS tree of $D$ with root $v_0$. For each $v \in V(D)$,
 let $c(v)$ be the distance of $v$ from $v_0$ in $T$ modulo $k$. We claim that,
 $c$ is a proper $k$-coloring of $D$. Let $C$ be an arbitrary directed cycle of 
 $D$. By \cite[Lemma $22.11$]{CLRS}, $C$ contains a back-edge $xy$. Let $P$ be the unique directed path
 from $y$ to $x$ in $T$. If $c(x)=c(y)$, then the length of $P$ is divisible
 by $k$. So, $D$ contains a cycle of length $1$ modulo $k$, a contradiction.
 Thus, $\chi(D) \leq k$.
\end{proof}

\section{Dichromatic Polynomial}
In [5] the dichromatic polynomial of digraphs was defined. Now, 
we want to discuss about some properties of dichromatic polynomial
for digraphs and tournaments.

\begin{thm}
Let $g$ be the girth of digraph $D$. Then in $P(D;x)$ the coefficients of 
$x^{n-1},\ldots,x^{n-g+2}$ are zero and the coefficient of 
$x^{n-g+1}$ is the number of cycles of length $g$ with negative sign.
\end{thm}
\begin{proof}
Let $c_1,\ldots,c_k$ be the minimal cycles of $D$. 
By the Inclusion-exclusion principle, the number of ways for 
coloring the vertices of $D$ with $x$ colors without a monochromatic cycle, 
is equal to $x^n$ minus the number of way with at least one monochromatic cycle, 
plus the number of way with at least two monochromatic cycles, etc. 
and except $x^n$ each other term is of the form $x^k$ with $k\leq n-g+1$. 
Furthermore there is exactly $g$ number of $x^{n-g+1}$ with negative sign. 
So the proof is complete.
\end{proof}
 \begin{cor}
 Let $T$ be a tournament with $n$ vertices. Then in $P(T;x)$ the coefficient of $x^{n-1}$  is zero, and the coefficient of $x^{n-2}$ is the number of directed triangles of $T$ with negative sign.
 \end{cor}
 
 In tournaments, we know the maximum number of directed triangles and some other theorems about the directed cycles, so maybe we can compute and compare the dichromatic polynomials for some tournaments.
 For example let $D$ be an acyclic tournament on the vertex set
 $\{v_1,\ldots,v_n\}$, such that the out-neighbors of $v_i$ are 
 $v_1,\ldots,v_{i-1}$. Now, reverse the edges of the Hamiltonian path
 $v_n,v_{n-1},\ldots,v_1$. The resulting tournament is denoted by $S_n$.

\begin{thm}
 For every positive integer $n$,
 $$P(S_n;x)=\sum_{i=1}^n\binom{i}{n-i}x(x-1)^{i-1}$$
\end{thm}

\begin{proof}
 Let $f_n(x)=P(S_n;x)$.
 It can be easily checked that $f_1(x)=x$ and $f_2(x)=x^2$. Let $n\geq3$. 
 The proper $k$-colorings of $S_n$ can be divided
 into two types: the colorings in which the colors of $v_n$ and $v_{n-1}$ are 
 different, and ones in which $v_n$ and $v_{n-1}$ have the same color.
 To get a coloring of the first type, it suffices to properly color the induced
 subgraph on $\{v_1,\ldots,v_{n-1}\}$, which is $S_{n-1}$, by $k$ colors 
 and then color $v_n$ by a color different from the color of $v_{n-1}$. 
 Thus the number of colorings of this type is $(k-1)f_{n-1}(k)$.
 In a similar way, for obtaining a coloring of the second type, one can properly 
 color the induced subgraph on $\{v_1,\ldots,v_{n-2}\}$,
 which is $S_{n-2}$, by $k$ colors and then color $v_{n-1}$
 and $v_{n}$ by a color that is different from that of $v_{n-2}$ (note that 
 $v_{n-2},v_{n-1},v_n$ form a cycle). Thus the number of the colorings of the
 second type is $(k-1)f_{n-2}(k)$. Therefore the number of proper $k$-colorings
 of $S_n$ is $(k-1)(f_{n-1}(k)+f_{n-2}(k))$ and we have the following 
 recurrence relation for $n\geq3$,
 
 $$f_n(x)=(x-1)\left(f_{n-1}(x)+f_{n-2}(x)\right).$$
 
 To solve this recurrence, we use the method of the generating functions. Define
 $$T(x,y)=\sum_{n=1}^{\infty}f_n(x)y^n$$
 Using the recurrence relation and initial conditions, we have
 \begin{eqnarray*}
 T(x,y) &=& xy+x^2y^2+(x-1)\sum_{n=3}^{\infty}
 \left(f_{n-1}(x)+f_{n-2}(x)\right)y^n\\
        &=& xy+x^2y^2+(x-1)y(T(x,y)-xy)+(x-1)y^2T(x,y)
 \end{eqnarray*}
 
 which gives
 
 $$T(x,y)=\frac{xy(y+1)}{1-(x-1)y(y+1)}.$$
 
 Now, Since
 
 $$\frac{1}{1-(x-1)y(y+1)}=\sum_{i=0}^{\infty}(x-1)^iy^i(y+1)^i,$$
 
 we get
 
 $$T(x,y)=x\sum_{i=1}^{\infty}(x-1)^{i-1}y^i(y+1)^i$$
 
 and the coefficient of $y^n$ in $T(x,y)$, which is $f_n(x)$, equals
 
 $$\sum_{i=1}^n\binom{i}{n-i}x(x-1)^{i-1},$$
 as desired.
\end{proof}

Next, we introduce the unique strongly connected tournament which has 
the maximum number of proper $k$-colorings for each $k>1$, among
all tournament of order $n$.
Let $D_n$ be the tournament with the vertex set 
$\{v_1,\ldots,v_n\}$ such that for $1\leq i<j\leq n$, where $i\neq1$
or $j\neq n$, the edge
between $v_i$ and $v_j$ is oriented from $v_i$ to $v_j$, and the edge between
$v_1$ and $v_n$ is oriented from $v_n$ to $v_1$. Clearly, $D_n$ is strongly
connected.

For a digraph $D$, $u,v\in V(D)$ and a positive integer $k$, 
let $P_{u\approx v}(D;k)$ be the
number of proper $k$-colorings of $D$ such that $u$ and $v$ have
the same colors, and $P_{u\not\approx v}(D;k)$ be the number of proper 
$k$-colorings of $D$ such that $u$ and $v$ have different colors.

The next lemma is an obvious observation about the chromatic polynomial of $D_n$.
\begin{lem}\label{Dnpoly}
 For two positive integers $n,k$, $P_{v_1\approx v_n}(D_n;k)=k(k-1)^{n-2}$,
 $P_{v_1\not\approx v_n}(D_n;k)=k^{n-1}(k-1)$. Moreover
 $P(D_n;k)=k(k-1)^{n-2}+k^{n-1}(k-1)$.
\end{lem}

\begin{lem}\label{allcycle}
 Let $T$ be a strongly connected tournament which has an edge $e=uv$ such
 that every directed cycle of $T$ contains $e$. Then $T=D_n$, $u=v_n$ and $v=v_1$.
\end{lem}

\begin{proof}
 Let $T'=D-e$. By assumption, $T'$ is an acyclic digraph, so
 there is a topological ordering of the vertices of $T'$
 like $v_1,\ldots,v_n$. Since $T$ is a strongly 
 connected tournament, $e$ should be an edge from $v_n$ to $v_1$
 which implies that $T=D_n$.
\end{proof}

\begin{lem}\label{Puv}
 Let $T$ be a strongly connected digraph of order $n$ and $u,v$ be two
 distinct vertices of $T$.
 Then for every positive integer $k$, 
 $P_{u\not\approx v}(T;k)\leq k^{n-1}(k-1)$, and the equality holds 
 if and only if $T=D_n$, $u=v_n$ and $v=v_1$.
\end{lem}

\begin{proof}
 Clearly $P_{u\not\approx v}(T;k)\leq k^{n-1}(k-1)$. 
 Now, suppose that there exist $k>1$ such that
 $P_{u\not\approx v}(T;k)\leq k^{n-1}(k-1)$. 
 If every cycle of $T$ contains $uv$, then
 by Lemma \ref{allcycle}, $T=D_n$, $u=v_n$ and $v=v_1$. Assume on the contrary
 that there is a cycle in $T$ not containing $uv$. It can be seen that
 there is a cycle $C$ in $T$ that does not contain either $u$ or $v$. Let
 $v\notin V(C)$ and $|V(C)|=r$. The number of proper $k$-colorings of $C$
 are $k^r-k$ and there are at most $(k-1)k^{n-r-1}$ to extend
 a $k$-proper coloring of $C$ to $T$.. Therefore,
 $$k^{n-1}(k-1)=P_{u\not\approx v}(T;k)\leq (k^r-r)(k-1)k^{n-r-1},$$
 which implies that $k^r\leq k^r-k$, a contradiction.
\end{proof}

\begin{thm}
 Let $T\neq D_n$ be a strongly connected tournament and $k>1$ be an integer.
 Then $P(T;k)<P(D_n;k)$.
\end{thm}

\begin{proof}
 The proof is by induction on $n$. For $n=1,2,3$ there is no strongly connected
 tournament of order $n$ other than $D_n$. So let $n\geq 4$.
 Since $T$ is strongly connected, there is a cycle $C$ of length $n-1$ in $T$.
 Let $V(T)\setminus V(C)=\{v\}$. Since $T$ is strongly connected,
 there is a triangle $vuw$ containing $v$. Now,
 
 $$P_{u\approx w}(T;k)\leq (k-1)P_{u\approx w}(T-v;k),$$
 $$P_{u\not\approx w}(T;k)\leq kP_{u\not\approx w}(T-v;k).$$
 
 Hence we find that
 $$P(T;k)\leq (k-1)P(T-v;k)+P_{u\not\approx w}(T-v;k).$$
 By the induction hypothesis and Lemma \ref{Dnpoly} we have,
 $$P(T-v,k)\leq P(D_{n-1},k)=k(k-1)^{n-3}+k^{n-2}(k-1),$$
 and by Lemma \ref{Puv} one can see that,
 $$P_{u\not\approx v}(T-v;k)\leq k^{n-2}(k-1).$$ 
 So the following holds:
 $$P(T;k)\leq k(k-1)^{n-2}+k^{n-1}(k-1)=P(D_n;k).$$
 
 Now, let $P(T;k)=P(D_n;k)$. Thus
 
 $$P_{u\approx w}(T;k)=(k-1)P_{u\approx w}(T-v;k),$$
 $$P_{u\not\approx v}(T-v;k)=k^{n-2}(k-1).$$ 
 
 Lemma \ref{Puv} and the last equality imply that $T-v$ is $D_{n-1}$, $w=v_1$
 and $u=v_{n-1}$. By the first equality, each proper $k$-coloring of $T-v$ in 
 which $u$ and $w$ have the same color $c$ can be extended to a coloring of $T$
 by assigning an arbitrary color different from $c$ to $v$. 
 Suppose that there is a 
 triangle $vu'w'$ such that $u'\neq u$ and $w'\neq w$. Choose a proper
 $k$-coloring of $T-v$ in which $u$ and $w$ have color $1$
 and $u'$ and $w'$ have color $2$. To extend this coloring to a proper
 $k$-coloring for $T$, we have at most $k-2$ choices for the color of $v$, 
 contrary to the assumption.
 Thus there is no such triangle and it can be easily seen that $T=D_n$.
\end{proof}


\end{document}